\documentclass[11pt]{amsart}
\usepackage{amsmath,amssymb,latexsym,dsfont}
\usepackage[left=2.6cm,right=2.6cm,top=2.7cm,bottom=2.7cm]{geometry}

\usepackage{amsmath,amsfonts,amssymb,epsfig, textcomp}
\usepackage{graphicx}
\usepackage{hyperref}
\usepackage{cases}
\usepackage{color}

\newtheorem{theorem}{Theorem}[section]
\newtheorem{lemma}{Lemma}[section]

\newtheorem{proposition}{Proposition}[section]

\newtheorem{remark}{Remark}[section]

\newtheorem{definition}{Definition}[section]
\numberwithin{equation}{section}

      \newcommand{\tw}{\widetilde w}

   \newcommand{\tW}{\widetilde W}

      \newcommand{\R}{{\mathbb{R}}}

      \newcommand{\dive}{\operatorname{div}}

      \newcommand{\loc}{\operatorname{loc}}
      
      \newcommand{\eps}{\varepsilon}
      \newcommand{\mR}{\mathbb{R}}

      \newcommand{\supp}{\operatorname{supp}}
      \newcommand{\dsp}{\displaystyle}

      \makeatletter
      \def\@setcopyright{}
      \def\serieslogo@{}
      \makeatother

\begin{document}

    \author[H.-M. Nguyen]{Hoai-Minh Nguyen}
\author[T. Nguyen]{Tu Nguyen}

\address[H.-M. Nguyen]{Department of Mathematics, EPFL SB CAMA, Station 8,  \newline\indent
	 CH-1015 Lausanne, Switzerland.}
\email{hoai-minh.nguyen@epfl.ch}

\address[T. Nguyen]{Institute of Mathematics, Vietnam Academy of Science and Technology,
	18 Hoang Quoc Viet, Cau Giay, Hanoi, Vietnam}
\email{natu@math.ac.vn}
\thanks{The second author is funded by the Vietnam National Foundation for Science and Technology Development (NAFOSTED) under grant number 101.02-2015.21}

 \title[Approximate cloaking for the heat equation via transformation optics]{Approximate cloaking for the heat equation via transformation optics}

  \maketitle

%



   \author{}
   \address{}
   \email{}


   \author{}

   \address{}

   \curraddr{}

   \email{}
   


%

   \dedicatory{}

   \date{\today}




\begin{abstract} In this paper, we establish approximate cloaking for the heat equation via transformation optics. We show that the degree of visibility is of the order $\eps$ in three dimensions and $|\ln \eps|^{-1}$ in two dimensions, where $\eps$ is the regularization parameter. To this end,
we first transform the problem in time domain into  a family of problems in frequency domain by taking the Fourier transform with respect to time, and 
then derive appropriate estimates in the frequency domain. 
\end{abstract}



\section{Introduction and statement of the results}

Cloaking using transformation optics (changes of variables) was introduced by Pendry, Schurig, and Smith \cite{Pendry} for the Maxwell system and by Leonhardt \cite{Leo} in the geometric optics setting.  These authors used a singular change of variables, which blows up a point into a
cloaked region. The same transformation had been used 
 to establish (singular) non-uniqueness in Calderon's problem in \cite{Green}. To avoid using
the singular structure, various regularized schemes have been proposed. One of them  was suggested by Kohn, Shen, Vogelius, and Weinstein  \cite{Kohn},  where instead of a point, a small ball  of radius $\eps$ is blown up to the cloaked region.  Approximate cloaking for acoustic waves  has been studied  in the quasistatic regime  \cite{Kohn, Ng-Vogelius-1},    the time harmonic regime \cite{ Kohn1, Ng-1, Ng-Vogelius-2, Ng-2},  and the time regime \cite{Ng-Vogelius-3, Ng-Vogelius-4},  and approximate cloaking for  electromagnetic waves has been studied in  the time harmonic regime  \cite{Ammari13, Lassas, MinhLoc}, see also the references therein. Finite energy solutions for the singular scheme  have been studied  extensively  \cite{GKLU07-1, Weder1, Weder2}. There are also other ways to achieve cloaking effects, such as the use of  plasmonic coating \cite{AE05},  active exterior sources \cite{VMO},  complementary media \cite{LaiChenZhangChanComplementary, Ng-Negative-Cloaking}, or via localized resonance  \cite{Ng-CALR-O} (see also \cite{MiltonNicorovici, Ng-CALR}).

The goal of this paper is to investigate approximate cloaking for the the heat equation using transformation optics.  Thermal cloaking via transformation optics was initiated by Guenneau, Amra, and Venante  \cite{GAV}. Craster, Guenneau, Hutridurga, and Pavliotis  \cite{CG} investigate 
the approximate cloaking for the heat equation using the approximate scheme in the spirit of \cite{Kohn}. They show that for the time large enough, the largeness depends on $\eps$, the degree of visibility is of the order $\eps^d$ ($d = 2, 3$) for  sources that are  independent of time. Their analysis is first based on the fact that as time goes to infinity, the solutions converge to the stationary states and then uses known results on approximate cloaking in the quasistatic regime \cite{Kohn, Ng-Vogelius-1}. 

  In this paper, we show that approximate cloaking is achieved at any positive time and established the degree of invisibility of order $\eps$ in three dimensions and $|\ln \eps|^{-1}$ in two dimensions. Our results hold for a general source that  depends on both time and space variables,  and our estimates are independent of  the content of the materials inside the cloaked region. The degree of visibility obtained herein is optimal due to the fact that a finite time interval  is considered (compare with \cite{CG}).


We next describe the problem in more detail  and state the main result. Our starting point is the regularization  scheme  \cite{Kohn} in which a transformation blows  up a small ball $B_\eps$ ($0< \eps < 1/2$) instead of a point into the cloaked region $B_1$ in $\mR^d$ ($d=2, 3$). Here and in what follows, for $r > 0$,  $B_r$ denotes the ball centered at the origin and of radius $r$ in $\mR^d$. Our assumption on the geometry of the cloaked region is mainly to  simplify  the notations. 
Concerning the transformation, we consider the map  $F_{\eps}: \R^d \rightarrow \R^d$ defined by
\begin{align}\label{def-Frho} F_{\eps} (x)= \left\{ \begin{array}{cl} x &\text{ in } \R^d \setminus B_2, \\[6pt] 
\dsp  \left( \frac{2 - 2\eps}{2-\eps} + \frac{|x|}{2-\eps} \right)\frac{x}{|x|} &\text{ in } B_2 \setminus B_{\eps},\\[6pt]
\dsp \frac{x}{\eps} &\text{ in } B_{\eps}.\end{array} \right. \end{align}
In what follows, we use the standard notations
\begin{equation}\label{pushforward}
F_*A(y) = \frac{\nabla F (x) A(x) \nabla F^T(x)}{ | \det \nabla F(x) |}, \quad F_*\rho (y) = \frac{\rho(x)}{ | \det \nabla F(x) |}, \quad x = F^{-1}(y),
\end{equation}
for the ``pushforward" of a symmetric, matrix-valued function $A$, and a scalar function  $\rho$,  by the diffeomorphism $F$,  and $I$ denotes the identity matrix. 
The cloaking device  in the region $B_2 \setminus B_1$  constructed from the transformation technique  is given  by 
\begin{equation}\label{TO--S}
\big( {F_{\eps}}_*I, {F_{\eps}}_*1 \big) \mbox{ in } B_2 \setminus B_1, 
\end{equation}
a pair of a matrix-valued function  and a function that characterize the material properties in $B_2 \setminus B_1$. Physically, $A$ is the thermal diffusivity  and $\rho$ is the mass density of the material. 

Let  $\Omega$ with $B_2 \subset \subset \Omega \subset \mR^d$ ($ d =2, 3$)\footnote{The notation $D \subset \subset \Omega$ means that the closure of $D$ is a subset of $\Omega$.}  be a bounded  region for which the  heat flow is considered. 
Suppose that the medium outside $B_2$ (the cloaking device and the cloaked region) is homogeneous so that it is characterized by the pair $(I, 1)$,  and the cloaked region is characterized by a pair $(a_O, \rho_O)$ where 
$a_O$ is a matrix-valued function and $\rho_O$ is a real function, both defined in $B_{1}$.  The medium in the whole space is then given by  
\begin{equation}\label{med-CH}
(A_c, \rho_c) = \left\{\begin{array}{cl} (I, 1) & \mbox{ in } \mR^d \setminus B_2, \\[6pt]
\big( {F_{\eps}}_*I, {F_{\eps}}_*1 \big)  & \text{ in } B_2 \setminus  B_1,\\[6pt]
(a_O, \rho_O)  &\text{ in } B_{1/2}.
\end{array} \right. 
\end{equation}
In what follows,  we make the usual assumption that  $a_O$ is symmetric and uniformly elliptic, i.e., for a.e. $x \in B_{1/2}$, 
\begin{equation}\label{elliptic}
\Lambda^{-1}|\xi|^2 \le \langle a_O(x)\xi, \xi \rangle \le \Lambda|\xi|^2\quad  \mbox{ for all } \, \xi \in \mR^d,
\end{equation}
for some $\Lambda \ge 1$, 
and $\sigma$ is a positive function bounded above and below by positive constants.  

For $0 < T \le \infty$, we denote 
$$
\Omega_T = (0, T) \times \Omega. 
$$
Given a function $f \in L^1 \big((0, + \infty), L^2(\Omega) \big)$ and an initial condition  $u_0 \in L^2(\Omega)$, in the medium characterzied by $(A_c, \rho_c)$, one obtains a unique weak solution $u_c \in L^2\big((0, \infty); H^1(\Omega)\big)$ \\ $\cap C \big( [0,+ \infty), L^2(\Omega) \big)$ of the equation 
\begin{equation}\label{eq-uc-1}
\left\{\begin{array}{cl}
\partial_t (\rho_c u_c ) - \dive (A_c \nabla u_c) = f  & \mbox{ in }  \Omega_\infty, \\[6pt]
u_c = 0 & \mbox{ on } (0, +\infty) \times \partial \Omega, \\[6pt]
u_c (t = 0, \cdot ) = u_0 & \mbox{ in } \Omega, 
\end{array}\right. 
\end{equation}
and in the homogeneneous medium characterized by $(I, 1)$, one gets a unique weak solution $u \in L^2\big((0, \infty); H^1(\Omega) \big) \cap C \big( [0,+ \infty), L^2(\Omega) \big)$  of the equation 
\begin{equation}\label{eq-u}
\left\{\begin{array}{cl}
\partial_t u - \Delta u  = f  & \mbox{ in }  \Omega_\infty, \\[6pt]
u_c = 0 & \mbox{ on } (0, +\infty) \times \partial \Omega, \\[6pt]
u_c (t = 0, \cdot ) = u_0 & \mbox{ in } \Omega. 
\end{array}\right. 
\end{equation}


\medskip 
The approximate cloaking meaning of the scheme \eqref{med-CH}  is given in  the following result:

\begin{theorem} \label{thm1} Let $u_0 \in L^2(\Omega)$ and  $f \in L^1\big( (0, + \infty); L^2(\Omega) \big)$ be such that $\supp f(t, \cdot) \subset \Omega \setminus B_2$ for $t>0$.  Assume that $u_c$ and $u$ are the solution of \eqref{eq-uc-1} and \eqref{eq-u} respectively. 
Then, for $0< \eps<1/2$,
	\begin{equation*}
	\| u_c(t) - u(t) \|_{H^1(\Omega \setminus B_2)} \le C e(\eps, d)  \Big( \|  f\|_{L^1_tL^2(\Omega)}   + \| u_0\|_{L^2(\Omega)} \Big), 
	\end{equation*}
for some positive constant $C$ independent of $f$, $u_0$, and $\eps$, where 
$$
e(\eps, d) = \left\{ \begin{array}{cl} \eps &  \mbox{ if } d = 3,  \\[6pt]
 |\ln \eps |^{-1} &  \mbox{ if } d = 2. 
 \end{array}\right.
$$
\end{theorem}

As a consequence of  Theorem~\ref{thm1}, $\lim_{\eps \to 0} u_c = u$ in $(0,T) \times (\Omega \setminus B_2)$ for all $f$ with compact support outside $(0, T) \times B_2$.  One therefore cannot detect the difference between $(A_c, \rho_c)$ and $(I, 1)$ as $\eps \rightarrow 0$ by observation of $u_c$ outside $B_2$: cloaking is achieved for observers outside $B_2$ in the limit as $\eps \rightarrow 0$.

We  now briefly describe  the idea of the proof. The starting point of the analysis is the invariance of the heat equations under a change of variables which we now state.  
\begin{lemma}\label{lem-TO-H}
 Let $d \ge 2$,  $T>0$, $\Omega$ be a bounded open subset of $\mR^d$ of class $C^1$, and let $A$ be an elliptic  symmetric matrix-valued function,  and $\rho$ be a bounded,  measurable   function defined on $\Omega$ bounded above and below by positive constants. Let $F: \Omega \mapsto \Omega$ be bijective such that $F$ and $F^{-1}$ are Lipschitz,  $\det \nabla F > c$ for a.e. $x \in \Omega$ for some $c>0$, and $F(x) = x$ near $\partial \Omega$. Let $f \in L^1\big( (0, T); L^2(\Omega) \big)$ and $u_0 \in L^2(\Omega)$.  Then $u \in L^2\big( (0, T); H^1_0(\Omega) \big) \cap C\big( [0,T), L^2(\Omega) \big)$ is the weak solution of
\begin{equation}\label{S1}
\left\{\begin{array}{cl}
\partial_{t} (\rho u)  - \dive (A \nabla u)  = f  & \mbox{ in } \Omega_T, \\[6pt]
 u = 0 & \mbox{ on }   (0, T) \times \partial \Omega, \\[6pt]
 u(0, \cdot) = u_0 & \mbox{ in } \Omega, 
\end{array}\right. 
\end{equation}
if and only if $v(t, \cdot): = u (t, \cdot) \circ F^{-1} \in  L^2\big((0, T); H^1_0(\Omega) \big) \cap C\big( [0,T), L^2(\Omega) \big)$ is the weak solution of
\begin{equation}\label{S2}
\left\{\begin{array}{cl}
\partial_{t} (F_*\rho v)  - \dive (F_*A \, \nabla v) = F_*f &  \mbox{ in } \Omega_T, \\[6pt]
 u = 0 & \mbox{ on }   (0, T) \times \partial \Omega, \\[6pt]
 v(0, \cdot) = u_0 \circ F^{-1} & \mbox{ in } \Omega.
\end{array}\right. 
\end{equation}
\end{lemma} 

Recall that  $F_*$ is defined in \eqref{pushforward}.  In this paper, we use the following standard definition of weak solutions: 

\begin{definition} \label{def1} Let $d \ge 2$. We say a function 
\begin{equation*}
u \in   L^2\big((0, T); H^1_0(\Omega) \big) \cap C \big( [0,T), L^2(\Omega) \big)
\end{equation*}
is a weak solution to \eqref{S1} if $u(0,\cdot)= u_0 \mbox{ in } \Omega$ and $u$ satisfies
\begin{align}\label{def-weak-sol}
\frac{d}{dt} \int_{\Omega} \rho u(t, \cdot) \varphi  + \int_{\Omega} A \nabla u(t, \cdot) \nabla \varphi = \int_{\Omega} f(t, \cdot) \varphi \mbox{ in } (0, T), 
\end{align}
in the distributional sense for all $\varphi \in H^1_0(\Omega)$. 
\end{definition}

The existence and uniqueness of weak solutions are standard, see, e.g., \cite{Allaire} (in fact, in \cite{Allaire}, $f$ is assumed in $L^2\big((0, T); L^2(\Omega) \big)$, however, the conclusion holds also for $f \in L^1\big((0, T); L^2(\Omega) \big) $ with a similar proof, see, e.g., \cite{Ng-Valentin}).  The proof of Lemma~\ref{lem-TO-H} is similar to that of the Helmholtz equation, see, e.g., \cite{Kohn1} (see also \cite{CG} for a parabolic version).

We now return to the idea of the proof of Theorem~\ref{thm1}. Set 
$$
u_\eps (t, \cdot) = u_c (t, \cdot) \circ F_\eps^{-1} \mbox{ for } t \in (0, +\infty). 
$$ 
Then $u_\eps$ is the unique solution of the system 
\begin{equation}\label{eq-uc-1-*}
\left\{\begin{array}{cl}
\partial_t (\rho_\eps  u_\eps ) - \dive (A_\eps \nabla u_\eps) = f  & \mbox{ in }  \Omega_\infty, \\[6pt]
u_\eps = 0 & \mbox{ on } (0, +\infty) \times \partial \Omega, \\[6pt]
u_\eps (t = 0, \cdot ) = u_0 & \mbox{ in } \Omega, 
\end{array}\right. 
\end{equation}
where
\begin{equation}\label{med-CH}
(A_\eps, \rho_\eps) = \left\{\begin{array}{cl} (I, 1) & \mbox{ in } \mR^d \setminus B_\eps, \\[6pt]
\Big( \eps^{2-d} a(\cdot/ \eps), \eps^{-d} \rho(\cdot / \eps) \Big) & \mbox{ in } B_\eps. 
\end{array} \right. 
\end{equation}
Moreover, 
$$
u_c - u = u_\eps - u \mbox{ in } (0, + \infty) \times B_2^c.  
$$

In comparing the coefficients of the systems verified by $u$ and $u_\eps$, the analysis can be derived from the study of the effect of a small inclusion $B_\eps$. The case in which  finite isotropic materials contain  inside the small inclusion  was investigated in \cite{AIKK} (see also \cite{Ams} for a related context). The analysis  in \cite{AIKK} partly involved the polarization tensor information and took the advantage of the fact that the coefficients inside the small inclusion are finite. In the cloaking context,   Craster et al.  \cite{CG} derived an estimate of the order $\eps^d$ for a time larger than a threshold one. Their analysis is based on 
long time behavior of solutions to parabolic equations and estimates for the degree of visibility of the conducting problem, see \cite{Kohn, Ng-Vogelius-1}, hence the threshold time goes to infinity as $\eps \to 0$.  

In this paper, to overcome the blow up of the coefficients inside the small inclusion and to achieve the cloaking effect at any positive time, 
we follow the approach of  Nguyen and Vogelius in \cite{Ng-Vogelius-3}.  The idea is to derive appropriate estimates for the effect of small inclusions in the time domain from the ones in the frequency domain using the Fourier transform with respect to time.   Due to the dissipative  nature of the heat equation, the problem in the frequency for the heat equation is more stable than the one corresponding to the acoustic waves, see,  e.g., \cite{Ng-Vogelius-2, Ng-Vogelius-3},   and the analysis is somehow easier to handle in the high frequency regime. After using a standard blow-up argument, a technical point in the analysis is to obtain an estimate for the solutions of  the equation $\Delta v + i \omega \eps^2 v = 0$ in $\mR^d \setminus B_1$ ($\omega > 0$) at the distance $1/\eps$ in which the dependence on $\eps$ and $\omega$ are explicit (see Lemma~\ref{Lem-Cor}). Due to the blow up of the fundamental solution in two dimensions, the analysis requires some new ideas.  We emphasize that even though our setting is in a bounded domain with zero Dirichlet  boundary condition, we employs Fourier transform in time instead of eigenmodes decomposition as in \cite{CG}.  This has the advantage that one can put both systems of $u_\eps$ and $u$ in the same context.

\section{Proof of the main result}

To implement the analysis in the frequency domain, let us introduce the Fourier transform with respect to time $t$: 
\begin{equation}\label{def-F}
\hat \varphi (k, x) = \frac{1}{\sqrt{2 \pi}} \int_{\mR} \varphi(t, x) e^{i k t} \, dt \mbox{ for } k \in \mR,  
\end{equation}
for $\varphi \in L^2((-\infty, + \infty), L^2(\mR^d))$.
Extending $u,\, u_c$, $u_\rho$,  and $f$ by $0$ for $t < 0$, and considering the Fourier with respect to time at the frequency $\omega > 0$, we obtain
\begin{equation*}
 \Delta \hat u  + i \omega \hat u= - ( \hat f  + u_0) \mbox{ in } \Omega, 
\end{equation*}
and 
\begin{equation*}
 \dive (A_\eps \nabla  \hat u_\eps) + i \omega \rho_\eps \hat u_\eps = - (\hat f + u_0) \mbox{ in } \Omega, 
\end{equation*}
where 
\begin{equation*}
(A_\eps, \rho_\eps)= \left\{ \begin{array}{cl} (I, 1)& \mbox{ in } \Omega \setminus B_\eps, \\[6pt]
\big(\eps^{2-d} a( \cdot/ \eps) , \eps^{-d}\sigma( \cdot /\eps) \big) & \mbox{ in } B_\eps. 
\end{array}\right. 
\end{equation*}

The main ingredient in the proof of Theorem~\ref{thm1} is the following: 

\begin{proposition}\label{main}  Let $\omega > 0$, $0 < \eps < 1/2$, and let $g \in L^2(\Omega)$ with $\supp g \subset \Omega \setminus B_2$. Assume that  $v, \, v_\eps \in H^1(\Omega)$ are respectively the unique solution of the systems 
\begin{equation*}
\left\{\begin{array}{cl}
 \Delta v  + i \omega v =  g &  \mbox{ in } \Omega, \\[6pt]
 v= 0 & \mbox{ on } \partial \Omega,  
\end{array}\right. 
\end{equation*}
and 
\begin{equation*}
\left\{\begin{array}{cl}
 \dive (A_\eps \nabla  v_\eps) + i \omega \rho_\eps v_\eps = g  & \mbox{ in } \Omega, \\[6pt]
 v_\eps = 0 & \mbox{ on } \partial \Omega. 
\end{array}\right. 
\end{equation*}
We have, for $4 \eps < r < 2$, 
\begin{equation}
\| v_\eps - v\|_{H^1(\Omega \setminus B_r)} \le C e(\eps, \omega, d) (1+\omega^{-1/2}) \| g \|_{L^2(\Omega)}, 
\end{equation}
for some positive constant $C = C_r$ independent of $\eps$, $\omega$, and $g$. 
Here 
\begin{equation}\label{er-3d}
e(\eps, \omega, 3) = \eps e^{-\sqrt{\omega}/4}, 
\end{equation}
and 
\begin{equation}\label{er-2d}
e(\eps, \omega, 2) = \left\{\begin{array}{cl} e^{-\sqrt{\omega}/4}/|\ln  \eps |& \mbox{ if } \omega \ge 1/2, \\[6pt]
\ln \omega /  \ln (\omega \eps) & \mbox{ if } 0 < \omega < 1/2. 
\end{array}\right. 
\end{equation}
\end{proposition}

The rest of this section is divided into three subsections. In the first subsection, we present several lemmas used in the proof of Proposition~\ref{main}. The proofs of Proposition~\ref{main}  and Theorem~\ref{thm1} are then given in the second and the third subsections,  respectively. 

\subsection{Preliminaries} In this subsection, we state and prove several useful lemmas used in the proof of Proposition~\ref{main}. Throughout, 
$D \subset B_1$ denotes a smooth, bounded, open subset of $\mR^d$ such that $\mR^d \setminus D$ is connected, and $\nu$ denotes the unit normal vector field on $\partial D$, directed into $\mR^d \setminus D$. The jump across $\partial D$ is denoted by $[ \cdot ]$.

\medskip 
The first result  is the following simple one:

\begin{lemma}\label{Lem1} Let  $d=2, \,  3$, $k > 0$,   and let $v \in H^1 (\R^d \setminus D)$ be such that   $\Delta v +ik v=0$ in $\mR^d \setminus D$. We have, for $R>2$,   
\begin{equation}\label{Lem1-a1}
\|v\|_{H^1(B_R\setminus D)} \le C_R (1+k) \|v\|_{H^{1/2}(\partial D)}, 
\end{equation}
for some positive constants $C_R$ independent of $k$ and $v$. 
\end{lemma}

\begin{proof} Multiplying the equation by $\bar v$ (the conjugate of $v$) and integrating by parts, we have 
$$
\int_{\mR^d \setminus D} |\nabla v|^2 - i k \int_{\mR^d \setminus D} |v|^2 = \int_{\partial D} \partial_\nu v \bar v. 
$$
This implies 
\begin{equation}\label{Lem1-p1}
\int_{\mR^d \setminus D} |\nabla v|^2 +  k \int_{\mR^d \setminus D} |v|^2 \le C  \| \partial_{\nu} v \|_{H^{-1/2}(\partial D)} \| v \|_{H^{1/2}(\partial D)}. 
\end{equation}
Here and in what follows,  $C$ denotes a positive constant independent of $v$ and $k$.  By the trace theory, we have
\begin{equation}\label{Lem1-p2}
\| \partial_{\nu} v \|_{H^{-1/2}(\partial D)} \le C \Big(\| \nabla v\|_{L^2(B_2 \setminus D)} + k \|v \|_{L^2(B_2 \setminus D)} \Big). 
\end{equation}
Combining \eqref{Lem1-p1} and \eqref{Lem1-p2} yields 
\begin{equation}\label{Lem1-p3}
\int_{\mR^d \setminus D} |\nabla v|^2 +  k \int_{\mR^d \setminus D} |v|^2 \le C (1+k)  \|  v \|_{H^{1/2}(\partial D)}^2 . 
\end{equation}
The conclusion follows when $k \ge 1$. 

Next, consider the case $0 < k < 1$. In the case where $d=3$, the conclusion  is a direct consequence of  \eqref{Lem1-p3} and the Hardy inequality (see, e.g.,  \cite[Lemma 2.5.7]{Nedelec}):
\begin{equation}\label{Hardy-3d}
\int_{\mR^3 \setminus D} \frac{|v|^2}{|x|^2} \le C \int_{\mR^3\setminus D } |\nabla v|^2.
\end{equation}

We next consider the case where $d=2$. 
It suffices to show 
\begin{equation}\label{Hardy} 
\int_{B_R \setminus D} |v|^2 \le C \| v\|^2_{H^{1/2}(\partial D)}.  
\end{equation}
By the Hardy inequality (see, e.g., \cite[Lemma 2.5.7]{Nedelec}), 
\begin{equation}\label{Hardy-2d}
\int_{\mR^2 \setminus D} \frac{|v|^2}{|x|^2 \ln (2 + |x|)^2} \le C \left( \int_{\mR^2\setminus D} |\nabla v|^2 + \int_{B_2 \setminus D} |v|^2 \right), 
\end{equation}
it suffices to  prove \eqref{Hardy} for $R = 2$. We proceed by contradiction. Suppose that there exists a sequence $(k_n) \to 0$ and a sequence $(v_n) \in H^1(\mR^2 \setminus D)$  such that 
\begin{equation*}
\Delta v_n + i k_n v_n = 0 \mbox{ in } \mR^2 \setminus D, \quad \| v_n\|_{L^2(B_2 \setminus D)} = 1, \quad \mbox{ and } \quad \lim_{n \to + \infty} \|v_n \|_{H^{1/2}(\partial D)} = 0. 
\end{equation*}
Denote 
\begin{equation*}
W^{1}(\mR^2 \setminus D) = \left\{ u \in L^1_{\loc}(\mR^2 \setminus D); \frac{u(x)}{\ln(2 + |x|) \sqrt{1 + |x|^2}}  \in L^2(\mR^2 \setminus D) \mbox{ and } \nabla u \in L^2(\mR^2 \setminus D) \right\}.
\end{equation*}
By  \eqref{Lem1-p3} and \eqref{Hardy-2d}, one might assume that $v_n$ converges to $v$ weakly in $H^1_{\loc}(\mR^2 \setminus D)$ and strongly in $L^2(B_2 \setminus D)$. Moreover, $v \in W^1(\mR^2 \setminus D)$, and satisfies 
\begin{equation}\label{Lem1-c1}
\Delta v = 0 \mbox{ in } \mR^2 \setminus D, \quad v = 0 \mbox{ on } \partial D, 
\end{equation}
and 
\begin{equation}\label{Lem1-c2}
 \| v \|_{L^2(B_2 \setminus D)} = 1. 
\end{equation}
From \eqref{Lem1-c1}, we have $v = 0$ in $\mR^2 \setminus D$ (see, e.g., \cite{Nedelec}) which contradicts \eqref{Lem1-c2}.  The proof is complete. 
\end{proof}

As a consequence of Lemma~\ref{Lem1}, we have 
\begin{lemma} \label{Lem-Cor} Let  $d=2, 3$, $\omega > 0$,  $0< \eps < 1/2$,  and let $v \in H^1 (\R^d \setminus D)$ be a solution of  $\Delta v +i\omega \eps^2 v=0$ in $\mR^d \setminus D$. We have, for $10 \eps  < r < R$ and $r < |x| < R$,   
\begin{equation}\label{Lem-Cor-p0}
| v(x/ \eps)|  \le C e(\eps, \omega, d) \|v\|_{H^{1/2}(\partial D)}, 
\end{equation}
for some positive constant $C= C_{r, R}$ independent of $\eps$,  $\omega$ and $v$. 
\end{lemma}

Recall that $e(\eps, \omega, d)$ is given in \eqref{er-3d} and \eqref{er-2d}. 

\begin{proof}  By Lemma~\ref{Lem1}, one might assume that $\eps < 1/16$.  By the representation formula, we have 
\begin{equation}\label{Lem1-representation}
	v (x ) = \int_{\partial B_3} \Big(G_{k} (x, y) \partial_r v (y) - \partial_{r_y} G_{k} (x, y) v(y)\Big) \, dy \mbox{ for } x \in
	\mR^3 \setminus B_4,  
\end{equation}
where
$$
G_k(x, y) = \frac{e^{ik |x-y|}}{4 \pi |x -y|} \mbox{ for } x \neq y  \mbox{ if } d = 3 \quad \mbox{ and } \quad G_k(x, y) =  \frac{i}{4}H^{(1)}_0 (k|x-y|)
\mbox{ if } d= 2,  
$$
with $k = e^{i \pi/4} \eps \omega^{1/2}. $ Recall that 
\begin{equation}\label{Lem-Cor-p1}
	H_0^{(1)} (z) = \frac{2i}{\pi}\ln\frac{|z|}{2} + 1  +\frac{2i\gamma}{\pi}+O(|z|^2\log|z|)\quad \text{ as } z\rightarrow 0, z\notin(-\infty,0], 
\end{equation}
and 	
\begin{equation}\label{Lem-Cor-p2}
	H_0^{(1)} (z) = \sqrt{\frac{2}{\pi z}} e^{i(z+\frac{\pi}{4})} (1+O(|z|^{-1})) \quad z\rightarrow \infty, z\notin(-\infty,0].
\end{equation}
The conclusion in  both the case where $d=3$ and the case where $d=2$ and $\omega > \eps^{-1/2} /2$ then follows from Lemma~\ref{Lem1}.

We next deal with  the case where $d=2$ and $\omega < \eps^{-1/2}/2$, hence $|k| < 1/2$.  Using \eqref{Lem1-representation}, we have, for $x \in \partial B_5$, 
\begin{equation*}
v(x) = \int_{\partial B_3} \Big( \big[G_{k} (x, y) - G_k(x, 0) \big]\partial_r v (y) - \partial_{r_y} G_{k} (x, y) v(y)\Big) \, dy + \int_{\partial B_3} G_k(x, 0)  \partial_r v (y)  \, dy. 
\end{equation*}
Using the fact 
$$
|\nabla_y G_k(x, y)| \le C |k| \mbox{ for } x \in \partial B_5 \mbox{ and } y \in B_3, 
$$
we derive that 
\begin{equation}\label{Lem1-est-vr}
\Big| \int_{\partial B_3}  \partial_r v (y)  \, dy \Big| \le C  \| g\|_{H^{1/2}(\partial D)} |\ln^{-1} |k| |. 
\end{equation}
Again using  \eqref{Lem1-representation}, we have, for $r < |x| < R$,  
\begin{equation*}
v(x/\eps) = \int_{\partial B_3} \Big( \big[G_{k} (x/\eps, y) - G_k(x/\eps, 0) \big]\partial_r v (y) - \partial_{r_y} G_{k} (x/\eps, y) v(y)\Big) \, dy + \int_{\partial B_3} G_k(x/ \eps, 0)  \partial_r v (y)  \, dy. 
\end{equation*}
Combining  \eqref{Lem-Cor-p1} and \eqref{Lem1-est-vr}, we obtain, for $0 < \omega < 1$,  
\begin{equation*}
|v(x/ \eps)| \le \frac{C |\ln \omega|}{|\ln|k||} \|g \| _{H^{1/2}(\partial D)}, 
\end{equation*}
and combining  \eqref{Lem-Cor-p2} and \eqref{Lem1-est-vr},  for $1 < \omega < \eps^{-1/2}/2$, we have 
\begin{equation*}
|v(x/ \eps)| \le C  |\ln |k||^{-1} e^{-\sqrt{\omega}/4} \|g \| _{H^{1/2}(\partial D)}.
\end{equation*}
The conclusion in the case where $d=2$ and $\omega < \eps^{-1/2}$ follows. 
The proof is complete. 
\end{proof}

\subsection{Proof of Proposition~\ref{main}} Multiplying the equation of $v_\eps$ by $\bar v_\eps$ and integrating in $\Omega$, we derive that 
\begin{equation}
\int_{\Omega} \langle A_\eps \nabla v_\eps, \nabla v_\eps \rangle +  \omega \int_{\Omega} \rho_\eps |v_\eps|^2 \le C \| g\|_{L^2(\Omega)}^2. 
\end{equation}
Here we used Poincar\'e's inequality 
$$
\| v_\eps \|_{L^2(\Omega)} \le C \| \nabla v_\eps\|_{L^2(\Omega)}. 
$$
In this proof, $C$ denotes a positive constant independent of $\eps$, $A_\eps$, $\rho_\eps$, $\omega$, and $g$.  It follows that  
\begin{equation}\label{SS}
\| v_{\eps} (\eps \cdot) \|_{H^{1/2}(\partial B_1)}^2 \le C \int_{B_1} |\nabla v_{ \eps}(\eps \cdot)|^2 + |v_{\eps} (\eps \cdot)|^2 \le C \int_{B_\eps} \frac{1}{\eps^{d-2}}|\nabla v_\eps|^2 +  \frac{1}{\eps^d} |v_\eps|^2  \le C (1+\omega^{-1}) \| g\|_{L^2(\Omega)}^2. 
\end{equation}
Let
$$
w_{\eps} =  v_{\eps} - v  \mbox{ in } \Omega \setminus B_\eps. 
$$
Then $w_{\eps} \in H^1(\Omega \setminus B_\eps)$ and  satisfies 
	\begin{equation}\label{comments}
	\left\{ \begin{array}{cl}
\Delta w_{\eps} + i \omega  w_{\eps} = 0  &\mbox{ in } \Omega \setminus B_\eps,  \\[6pt] 
w_{\eps} = v_\eps - v   &\mbox{ on }  \partial B_{\eps},  \\[6pt] 
 w_{\eps} = 0 &\mbox{ on } \partial \Omega.
\end{array}\right. 
\end{equation}
Let $\tw_{\eps}\in H^1(\mR^d \setminus B_{\eps})$ be the unique solution of the system 
\begin{equation}
\left\{ \begin{array}{cl}
\Delta \tw_{\eps} + i \omega  \tw_{\eps} = 0 & \mbox{ in } \mR^d \setminus B_{\eps}, \\[6pt]
 \tw_{\eps} =  w_\eps &  \mbox{ on } \partial B_{\eps}, 
\end{array}\right. 
\end{equation}
and set 
$$
\tW_{\eps}  = \tw_{\eps}(\eps \,  \cdot \, ) \mbox{ in } \mR^d \setminus B_1. 
$$
Then  $\tW_{\eps}  \in H^1(\mR^d \setminus B_{1})$ is the unique solution of the system 
\begin{equation}
\left\{ \begin{array}{cl}
\Delta \tW_{\eps} + i \omega \eps^2  \tW_{\eps} = 0 &  \mbox{ in } \mR^d \setminus B_{1}, \\[6pt]
 \tW_{\eps} = w_\eps (\eps \, \cdot \, )  & \mbox{ on } \partial B_{1}. 
\end{array}\right. 
\end{equation}

Fix $r_0 > 2$ such that $\Omega \subset B_{r_0}$.  By Lemmas \ref{Lem1} and  \ref{Lem-Cor}, we have, for 
$r/2  \le |x| <  r_0$, that  
\begin{equation*}
|\tW_{\eps} (x / \eps)|  \le  C e(\eps, \omega,  d) \|w_\eps (\eps \,  \cdot \, )  \|_{H^{1/2}(\partial B_1)}, 
\end{equation*}
which yields,  for $x \in B_{r_0} \setminus B_{r/2}$, that 
\begin{equation*}
|\tw_{\eps}(x)|
\le  C   e(\eps, \omega,  d) \|w_\eps (\eps \,  \cdot \, )  \|_{H^{1/2}(\partial B_1)}. 
\end{equation*}
Since $\Delta \tw_{\eps} + i \omega \tw_{\eps} = 0$ in $B_{r_0} \setminus B_{r/2}$, it follows that 
\begin{equation}\label{main-p1}
\|\tw_{\eps} \|_{H^1(\Omega \setminus B_{2r/3})}
\le  C  e(\eps, \omega,  d) \|w_\eps (\eps \,  \cdot \, )  \|_{H^{1/2}(\partial B_1)}. 
\end{equation}

Fix $\varphi \in C^2(\mR^d)$  such that  $\varphi = 1$ in $B_{2r/3}$ and $\varphi = 0$ in $\mR^d \setminus B_r$,  and set 
$$
\chi_{\eps} = w_{\eps} - \varphi \tw_{\eps} \mbox{ in } \Omega \setminus B_\eps. 
$$
Then $\chi_{\eps}\in H^1_0(\Omega \setminus B_\eps)$  and satisfies 
\begin{equation*}
 \Delta \chi_{\eps} +i\omega \chi_{\eps} = -\Delta \varphi \tw_{\eps} - 2\nabla \varphi \cdot \nabla \tw_{\eps}   \mbox{ in }  \Omega\setminus B_\eps.
\end{equation*}
Multiplying the equation of $\chi_{\eps}$ by $\bar \chi_{\eps}$ and integrating by parts, we obtain 
\begin{equation*}
\int_{\Omega \setminus B_\eps} |\nabla \chi_{\eps}|^2 \le C \| \tw_{\eps}\|_{H^1(\Omega \setminus B_{2r/3})}^2. 
\end{equation*}
This yields, by Poincar\'e's inequality, 
\begin{equation}\label{main-p2}
\| \chi_{\eps} \|_{H^1(\Omega \setminus B_\eps)}\le C \|\tw_{\eps} \|_{H^1(\Omega \setminus B_{2r/3})}.
\end{equation}
Combining \eqref{main-p1} and \eqref{main-p2} yields 
\begin{equation}\label{w1}
\| w_{\eps} \|_{H^1(\Omega \setminus B_r)}  \le   C e(\eps, \omega,  d) \|w_\eps (\eps \,  \cdot \, )  \|_{H^{1/2}(\partial B_1)}. 
\end{equation}
The conclusion now follows from \eqref{SS}. 
\qed 

\begin{remark} \rm  The estimate in Proposition~\ref{main} is  independent of the coefficients inside $B_\eps$
and is  optimal. In fact, one can choose the coefficients in $B_\eps$ such that $v_\eps$ on $\partial B_\eps$ is  as small as one wants.  
\end{remark}

\subsection{Proof of Theorem~\ref{thm1}} Let $v_\eps=u_\eps - u$. Using the fact that  $v_\eps$  is real, by the inversion theorem and Minkowski's inequality, we have, for $t > 0$,  
\begin{align}\label{ss}
\|v_\eps(t, \cdot) \|_{L^2(\Omega \setminus B_2)}  \le C \int_0^\infty \|\hat{v}_\eps (\omega, \cdot) \|_{L^2(\Omega \setminus B_2) } \, d \omega. 
\end{align}
Using Proposition~\ref{main}, we get  
\begin{align*}
 \int_0^\infty \|\hat{v}_\eps (\omega, \cdot) \|_{L^2(\Omega \setminus B_2) } \, d \omega  &\le C \int_0^\infty (1+\omega^{-1/2})e(\eps,\omega,d)\| \hat{f}(\omega)+u_0\|_{L^2(\Omega \setminus B_2) } \, d \omega \\[6pt]
& \le C  \mbox{esssup}_{\omega > 0} \|\hat{f}(\omega)+u_0\|_{L^2(\Omega \setminus B_2) }  \int_0^\infty  (1+\omega^{-1/2})e(\eps,\omega,d) \, d \omega \\[6pt]
& \le Ce(\eps,d) \big(\|f\|_{L^1\big( (0, + \infty);  L^2(\Omega) \big)} +\|u_0\|_{L^2(\Omega)} \big).  
\end{align*}
It follows from \eqref{ss} that, for $t > 0$,  
\begin{align*}
\|v_\eps(t, \cdot) \|_{L^2(\Omega \setminus B_2)}  \le Ce(\eps,d) \big(\|f\|_{L^1\big( (0, + \infty);  L^2(\Omega) \big)} +\|u_0\|_{L^2(\Omega)} \big).  
\end{align*}
Similarly, we have, for $t > 0$,  
\begin{equation*}
\|\nabla v_\eps(t, \cdot) \|_{L^2(\Omega \setminus B_2)}   \le Ce (\eps,d) \big(\|f\|_{L^1\big( (0, + \infty);  L^2(\Omega) \big)} +\|u_0\|_{L^2(\Omega)} \big).  
\end{equation*}
The conclusion follows. \qed


\begin{thebibliography}{99}



\bibitem{Allaire} G. Allaire, Numerical analysis and optimization. An introduction to mathematical modelling and numerical simulation.  Oxford University Press, Oxford, 2007. 
 
\bibitem{AE05} A. Alu, N. Engheta, \emph{Achieving transparency with plasmonic and metamaterial coatings}
Phys. Rev. E {\bf 72} (2005), 016623. 

\bibitem{AIKK} H. Ammari, E. Iakovleva, H. Kang, and K. Kim, \emph{Direct Algorithms for Thermal Imaging of Small Inclusions},  Multiscale Model. Simul., {\bf 4} (2005) 1116--1136. 


	 
\bibitem{Ammari13} H.~Ammari, H.~Kang, H.~Lee, M.~Lim, S.~Yu, \emph{Enhancement of Near Cloaking for the Full Maxwell Equations}, SIAM J. Appl. Math. {\bf  73} (2013), 2055--2076.
	
  	

\bibitem{Ams} S. Amstutz, T. Takahashi, B. Vexler, \emph{Topological sensitivity analysis for time-dependent problems}, 
ESAIM Control Optim. Calc. Var. {\bf 14} (2008),  427--455. 



\bibitem{CG} R. V. Craster, S.  Guenneau, H.   Hutridurga, and G. Pavliotis, \emph{Cloaking via Mapping for the Heat Equation}, 	
Multiscale Model. Simul., {\bf 16} (2018), 1146--€"1174. 





\bibitem{GAV} S. Guenneau, C. Amra, and D.  Veynante,\emph{ Transformation thermodynamics: cloaking and concentrating heat flux}.  Opt. Express {\bf 20} (2012), 8207--8218. 


\bibitem{GKLU07-1} A.~Greenleaf, Y.~Kurylev, M.~Lassas, and G.~Uhlmann, 
 \emph{{Full-wave invisibility of active devices at all frequencies}},
  Comm. Math. Phys. \textbf{275} (2007), 749--789.

\bibitem{Green} A.~Greenleaf, M.~Lassas, G.~Uhlmann, \emph{On nonuniqueness for Calderon's inverse problem}, Math. Res. Lett. {\bf 10} (2003), 685--693.
		
  


%
%


	

\bibitem{Kohn}
R.~V.~Kohn, H.~Shen, M.S.~Vogelius, M.~I.~Weinstein, \emph{Cloaking via change of variables in electric impedance tomography}, Inverse Problem {\bf 24} (2008), 015--016.

	\bibitem{Kohn1}
R.~V.~Kohn, D.~Onofrei, M.~S.~Vogelius, M.~I.~Weinstein, \emph{Cloaking via change of variables for the Helmholtz equation},  Comm. Pure Appl. Math. {\bf 63} (2010), 
973--1016.

\bibitem{LaiChenZhangChanComplementary}
Y.~Lai, H.~Chen, Z.~Zhang, and C.~T. Chan, \emph{{Complementary media
  invisibility cloak that cloaks objects at a distance outside the cloaking
  shell}}, Phys. Rev. Lett. \textbf{102} (2009), 093901.
  
\bibitem{Lassas}
M.~Lassas, T.~Zhou, \emph{The blow-up of electromagnetic fields in 3-dimensional invisibility cloaking for Maxwell's equations}, SIAM J. Appl. Math. {\bf 76} (2016), 457--478.

\bibitem{Leo}
U.~Leonhardt, \emph{Optical conformal mapping}, Science {\bf 312} (2006), 1777--1780.




\bibitem{MiltonNicorovici}
G.~W. Milton and N.~A. Nicorovici, \emph{{On the cloaking effects associated
with anomalous localized resonance}}, Proc. R. Soc. Lond. Ser. A \textbf{462}
(2006), 3027--3059.



 
\bibitem{Nedelec}
J.~C.~N\'ed\'elec, \emph{Acoustic and electromagnetic equations, integral representations for harmonic problems}, Springer, 2000. 
  

  
  \bibitem{Ng-1}
  H-M.~Nguyen, \emph{Cloaking via change of variables for the Helmholtz equation in the whole space}, Com. Pure  Appl. Math. {\bf 63} (2010), 1505--1524.

  \bibitem{Ng-2}
  H-M.~Nguyen, \emph{Approximate cloaking for the Helmholtz equation via transformation optics and consequences for perfect cloaking}, Comm. Pure  Appl. Math. {\bf 65} (2012), 155--186. 
  
  



  \bibitem{Ng-CALR}
H-M. Nguyen, \emph{{Cloaking via anomalous localized resonance for doubly
complementary media in the quasistatic regime}}, J. Eur. Math. Soc. (JEMS)
  \textbf{17} (2015), 1327--1365.

  
  \bibitem{Ng-Negative-Cloaking}
H-M. Nguyen, \emph{{Cloaking using complementary media in the quasistatic regime}},
  Ann. Inst. H. Poincar\'e Anal. Non Lin\'eaire \textbf{33} (2016), 1509--1518. 

   \bibitem{Ng-CALR-O}
  H-M.~Nguyen, \emph{Cloaking an arbitrary object via anomalous localized resonance: the cloak is independent of the object},  SIAM J. Math. Anal. 
  {\bf 49} (2017) 3208--3232.






  \bibitem{MinhLoc}
H-M.~Nguyen, X. L. Tran, \emph{Approximate cloaking for electromagnetic waves via transformation optics: cloaking vs infinite energy}, submitted, 2018. 


  \bibitem{Ng-Valentin}
   H.-M. Nguyen and  V. Vinoles, \emph{Electromagnetic wave propagation in media consisting of dispersive metamaterials},  C. R. Math. Acad. Sci. Paris {\bf 356} (2018), 757--775.


 

 \bibitem{Ng-Vogelius-1}
H-M.~Nguyen, M.~S.~Vogelius, \emph{A representation formula for the voltage perturbations caused by diametrically small conductivity inhomogeneities. Proof of uniform validity}, Ann. Inst. H. Poincar\'e Anal. Non Lin\'eaire \textbf{26} (2009), 2283--2315. 


  
  
\bibitem{Ng-Vogelius-2}
H-M.~Nguyen, M.~S.~Vogelius, \emph{Full range scattering estimates and their application to cloaking}, Arch. Rational Mech. Anal. {\bf 203} (2012), 769--807.

  \bibitem{Ng-Vogelius-3}
  H-M.~Nguyen, M.~S.~Vogelius, \emph{Approximate cloaking for the wave equation via change of variables}, SIAM J. Math. Anal. {\bf 44} (2012), 1894--1924.
  
  \bibitem{Ng-Vogelius-4}
  H-M.~Nguyen, M.~S.~Vogelius, \emph{Approximate cloaking for the full wave equation via change of variables: the Drude-Lorentz model},  J. Math. Pures Appl. {\bf 106} (2016),  797--836. 
  
  
  \bibitem{Pendry}
  J.~B.~Pendry, D.~Schurig, D.~R.~Smith \emph{Controlling electromagnetic fields}, Science {\bf 321} (2006), 1780-1782.
 


\bibitem{VMO} F. G.  Vasquez, G. W. Milton, D. Onofrei \emph{Active exterior cloaking for the 2D Laplace and Helmholtz equations}
 Phys. Rev. Lett. {\bf 103} (2009), 073901.
  
\bibitem{Weder1}
R.~Weder, \emph{{A rigorous analysis of high-order electromagnetic invisibility
  cloaks}}, J. Phys. A: Math. Theor. \textbf{41} (2008), 065207.

\bibitem{Weder2}
R.~Weder, \emph{{The boundary conditions for point transformed electromagnetic
  invisibility cloaks}}, J. Phys. A: Math. Theor. \textbf{41} (2008), 415401.

\end{thebibliography}
\end{document}